\documentclass[12pt, reqno]{amsart}

\usepackage[bookmarksnumbered, colorlinks, plainpages]{hyperref}
\usepackage{amsmath, amsthm, amscd, amsfonts, amssymb, graphicx, color}

\newtheorem{theorem}{Theorem}[section]
\newtheorem{lemma}[theorem]{Lemma}
\newtheorem{proposition}[theorem]{Proposition}
\newtheorem{corollary}[theorem]{Corollary}
\theoremstyle{definition}
\newtheorem{definition}[theorem]{Definition}

\theoremstyle{remark}
\newtheorem{remark}[theorem]{Remark}
\numberwithin{equation}{section}

\newcommand\G{\mathbb{G}}

\begin{document}
\setcounter{page}{1}

\title[$L^2(\G)$ approach to approximate diagonals]{A Hilbert space approach to approximate diagonals for locally compact quantum groups}

\author[B. Willson]{Benjamin Willson$^1$$^{*}$}

\address{$^{1}$ Department of Mathematics, School of Natural Sciences, Hanyang University, 222 Wangsimni-ro, Seongdong-gu, Seoul 133-791, Korea.}
\email{\textcolor[rgb]{0.00,0.00,0.84}{bwillson@hanyang.ac.kr;bwillson@ualberta.ca}}




\date{Received: xxxxxx; Revised: yyyyyy; Accepted: zzzzzz.
\newline \indent $^{*}$ Corresponding author}

\begin{abstract}
For a locally compact quantum group $\G$, the quantum group algebra $L^1(\G)$ is operator amenable if and only if it has an operator bounded approximate diagonal.  It is known that if $L^1(\G)$ is operator biflat and has a bounded approximate identity then it is operator amenable.  In this paper, we consider nets in $L^2(\G)$ which suffice to show these two conditions and combine them to make an approximate diagonal of the form  $\omega_{{W'}^*\xi\otimes\eta}$ where $W$ is the multiplicative unitary and $\xi\otimes\eta$ are simple tensors in $L^2(\G)\otimes L^2(\G)$.  Indeed, if $L^1(\G)$ and $L^1(\hat{\G})$ both have a bounded approximate identity and either of the corresponding nets in $L^2(\G)$ satisfies a condition generalizing quasicentrality then this construction generates an operator bounded approximate diagonal.  In the classical group case, this provides a new method for constructing approximate diagonals emphasizing the relation between the operator amenability of the group algebra $L^1(G)$ and the Fourier algebra $A(G)$.
\end{abstract} \maketitle

\section{Introduction}

\noindent 
A locally compact group $G$ is amenable if and only if the group algebra $L^1(G)$ is an amenable Banach algebra.  Johnson\cite{johnson} showed this (and defined amenability for Banach algebras) by using a bounded approximate identity and a net of functions tending to left invariance (or tending to invariance under multiplication by positive norm $1$ elements of $L^1(G)$).  Ruan \cite{ruan} proved that the Fourier algebra of $G$ is operator amenable (has an operator approximate diagonal) if and only if $G$ is amenable.  Again, this can be shown by combining a bounded approximate identity for $A(G)$ and a net of (quasi-central) functions tending to invariance under multiplication in $A(G)$.  

In \cite{kv2003}, Kustermans and Vaes formalize the notion of a locally compact quantum group as a von Neumann algebra with a co-multiplication and left and right Haar weights.  The predual of this von Neumann algebra is a Banach algebra which is analogous to both $L^1(G)$ and $A(G)$ in the group setting.  Indeed, there is a `quantum' version of the Pontryagin duality theorem which extends the duality between $L^1(G)$ and $A(G)$.    There are various notions of amenability which generalize to the quantum group setting.  It is natural to ask whether classical results about the equivalence of these extends to quantum groups. In particular, how are the amenability of a locally compact quantum group $\G$ and its dual $\hat{\G}$ related?

In section 2, we provide some background including the definitions of strong amenability and co-amenability.  These naturally dual conditions are characterized by the existence of certain nets in the underlying Hilbert space.  These nets are related to a left invariant net and a bounded approximate identity (resp.) in $L^1(\G)$. On the dual side, they are related to a b.a.i. and left invariant net in $L^1(\hat{\G})$.  

In \cite{ruanxu}, Ruan and Xu show that for a Kac algebra, the predual of the von Neumann algebra is operator amenable if it admits a bounded approximate identity and there is a net in the underlying Hilbert space tending to translation invariance and quasicentrality.  In section 3 of this paper, we extend their result to preduals of von Neumann algberas on locally 
compact quantum groups.  The current approach also constructs an operator approximate diagonal directly from the nets in $L^2(\G)$ rather than 
using machinery in the second dual of $L^1(\G)$.  The main result is that if $\G$ is strongly amenable and co-amenable and either of the nets satisfies a quasicentral condition then $L^1(\G)$ is operator amenable.  

In section 4, the analogous construction of an operator bounded approximate diagonal for $L^1(\hat{\G})$ is described.  The quasicentral conditions of the two nets are not dual to each other.  We show a result related to a dual version of this quasicentrality.

Other constructions of quasi-central approximate diagonals have been considered by Stokke in \cite{stokke}.  Nets tending to left invariance have been studied for semidirect products in \cite{willson2009PAMS}.The nets used throughout this paper are also related to operator biflatness of $L^1(\G)$.  It was conjectured by Aristov, Runde, and Spronk\cite{ars} that for a locally compact group, $A(G)$ is always operator biflat.  This remains an open question.

\section{Background}

We rely upon the paper of Kustermans and Vaes	\cite{kv2003} to provide further details on locally compact quantum groups and will herein provide only a brief introduction.  
We will use the notation from \cite{kv2003} with one significant adjustment.  We will emphasize the connection
to the classical results by referring to a locally compact quantum group as $\G=(M, \Gamma, \phi, \psi)$ and using $L^\infty(\G)=M$, 
$L^1(\G)=M_*$, and $L^2(\G)=\mathcal{H}_\phi$.  However, one should not infer from this notation that there is some set `$\G$' on which we define spaces of functions.

\begin{definition}
A {\em locally compact quantum group} $\G=(M, \Gamma, \phi, \psi)$ consists of a von Neumann algebra $M$ (or $L^\infty(\G)$), a comultiplication $\Gamma$, and left ($\phi$) and right ($\psi$) Haar weights.  

The unique (as a Banach space) predual of $L^\infty(\G)$ has a multiplication given by the pre-adjoint of $\Gamma$.  This makes the Banach algebra $L^1(\G)$ which will be referred to as the {\em quantum group algebra of $\G$}.  

The left Haar weight $\phi$ is used, via the GNS construction, to create the Hilbert space $L^2(\G)$.  Using this Hilbert space, we consider $L^\infty(\G)\hookrightarrow B(L^2(\G))$.  Furthermore, every $\zeta\in L^2(\G)$ can generate a $\omega_{\zeta}\in L^1(\G)$ via:
$
X(\omega_\zeta) = \langle X\zeta, \zeta\rangle
$
\end{definition}

We will also consider several tensor products of the above spaces:  $L^2(\G)\otimes_{2}L^2(\G)$ is the Hilbert space tensor product, $L^\infty(\G)\otimes_{\text VN}L^\infty(\G)$ is the von Neumann tensor product (in $B(L^2(\G)\otimes_2 L^2(\G))$), and $L^1(\G)\hat\otimes L^1(\G)$ is the operator space projective tensor product (see \cite{effrosruan} for details) which is the predual of $L^\infty(\G)\otimes_{\text VN}L^\infty(\G)$.
 
 There exists a unitary $W$ in $B(L^2(\G)\otimes L^2(\G))$  (in fact, $W$ is in $L^\infty(\G)\otimes_{VN}L^\infty(\hat{\G})$)such that $\Gamma(x) = W^*(1\otimes x) W$ for $x\in L^\infty(\G)$.  We call $W$  the left multiplicative unitary.  Similarly there is a right multiplicative unitary $V$ and 
using $V$ and $L^2(\G)$, we can reconstruct $L^\infty(\G)$ as the $\sigma$-strong closure of $\{(\omega\otimes\iota)(V)|	\omega\in B(L^2(\G))_*\}.$
The dual LCQG, $\hat{\G}$ is the $\sigma$-strong closure of $\{(\omega\otimes \iota)(W)|\omega\in B(L^2(\G))_*\}$  ( Here, $\{(\omega\otimes \iota)W$ denotes the slice map of $W$ with $\omega$ in the first component).

The left Haar weight gives rise to a modular conjugation $J$   for $L^2(\G)$ and there is a unique left invariant weight for $\hat{\G}$ that has a corresponding modular conjugation $\hat{J}$.

Along with $W$ and $V$, we can find similar unitaries in $B(L^2(\G)\otimes_2 L^2(\G))$ which correspond to related LCQGs.  In particular, there are $\hat{W}$ (the left multiplicative unitary for the dual LCQG), $W^{\text op}$ (for the opposite LCQG -- $\G$ with opposite co-multiplication), and $W'$ (for the commutant LCQG -- with von Neumann algebra $L^\infty(\G)'$).

\begin{proposition}\label{prop:list}
Among these operators, the following relations are satisfied:

\begin{align*}
W^* &= (\hat{J} \otimes J)W(\hat{J} \otimes J)\\
\hat{J}J&=\nu^{\frac{i}{4}}J\hat{J}\\
\hat{W} &=\Sigma W^* \Sigma\\
V&=(\hat{J}\otimes\hat{J})\Sigma W^*\Sigma(\hat{J}\otimes\hat{J})\\
\hat{V}&=(J\otimes J)W(J\otimes J)\\
W^{\text op} &= \Sigma V^* \Sigma = (\hat{J}\otimes\hat{J}) W(\hat{J}\otimes\hat{J})\\
W' &=\hat{V} = (J\otimes J)W(J\otimes J)\\
W_{12}W_{13}W_{23} &= W_{23}W_{12}\\
\hat{W}' &= \widehat{W^{\text{op}}}
\end{align*}
where $\nu$ is some positive number (the scaling constant of the quantum group), $\Sigma$ denotes the flip map for tensors, and the subscripts on $W$ are the leg notation (eg. $W_{12}=W\otimes 1\in B(L^2(\G)\otimes L^2(\G)\otimes L^2(\G)$).
\end{proposition}
\begin{proof}
See \cite{kv2003} 2.2, 2.12, 2.15, and 4.1.
\end{proof}
\begin{definition}
A locally compact quantum group is {\em{co-amenable}} if $L^\infty(\G)^*$ (with respect to the two Arens products) is unital.  Co-amenability is equivalent to the existence of bounded approximate identity for $L^1(\G)$.  B\'edos and Tuset\cite{bedostuset} showed that $\G$ is co-amenable if and only if there is a net of norm one vectors $(\eta_\beta)_\beta$ in $L^2(\G)$ such that for each $\zeta\in L^2(\G)$:
\[\tag{CA}
\lim_{\beta} \|W(\eta_\beta\otimes \zeta) - (\eta_\beta\otimes \zeta) \|_2= 0.
\]

Runde\cite{runde} remarked that such a net also satisfies the equivalent condition for the opposite LCQG $\G^{\text op}$:
\[
\lim_{\beta} \|W^{\text op}(\eta_\beta\otimes \zeta) - (\eta_\beta\otimes \zeta) \|_2= 0
\]
hence such a net generates a 2-sided bounded approximate identity for $L^1(\G)$.

\end{definition}

\begin{definition}
A locally compact quantum group is {\em{strongly amenable}} if $\hat{\G}$ is co-amenable.  Equivalently, $\G$ is strongly amenable if there is a net of norm one vectors 
$(\xi_\alpha)_\alpha$ in $L^2(\G)$ such that for each $\zeta\in L^2(\G)$:
\[
\tag{SA}
\lim_{\alpha} \|W(\zeta\otimes \xi_\alpha) - (\zeta\otimes\xi_\alpha) \|_2= 0.
\]
\end{definition}

\begin{remark}
There is another definition of amenability for locally compact quantum groups that depends on the existence of a translation invariant mean in the dual of $L^\infty(\G)$.  If $\G$ is strongly amenable, then it is amenable, but the converse is an open question.  (In the group case, amenability and strong amenability are equivalent - this is related to the equivalence of Reiter's P1 and P2 conditions).  See \cite{bedostuset, dawsrunde, DQV} for investigations into this question.
\end{remark}

\begin{remark}
Johnson's main result of 
\cite{johnson} was to show that a group is amenable precisely when the group algebra has a certain homological property which he also termed amenability (for Banach algebras).  It is possible to characterize this in terms of the existence of 
an approximate diagonal in the tensor product $L^1(G)\hat{\otimes} L^1(G)$.  Ruan\cite{ruan} extended Johnson's notion of amenability to completely contractive Banach algebras.  In particular, he showed that $G$ is amenable if and only if there is an operator bounded approximate diagonal for $A(G)$.   For LCQGs the appropriate operator space structure on $L^1(\G)$ is that resulting from considering the predual of the von Neumann algebra operator space structure of $L^\infty(\G)$.  This leads to the following characterization:
\end{remark}
\begin{definition}
The quantum group algebra $L^1(\G)$ is {\em{operator amenable}} if it admits an operator bounded approximate diagonal.  That is, if there is a net $(x_\gamma)_\gamma$ in $L^1(\G)\hat{\otimes}L^1(\G)$ such that for any $a\in L^1(\G)$:
\begin{align*}
\tag{OBAD1}\|a\cdot x_\gamma - x_\gamma\cdot a\|_{L^1(\G)\hat{\otimes}L^1(\G)}&\rightarrow 0, 
{\text{ and;}}\\
\tag{OBAD2}\|\Gamma_*(x_\gamma)a - a\|_{L^1(\G)}&\rightarrow 0.
\end{align*}

Here $\cdot$ denotes the natural bimodule action of $L^1(\G)$ on $L^1(\G)\hat{\otimes}L^1(\G)$ which is multiplication on the left in the first co-ordinate and multiplication on the right in the second.
\end{definition}

\section{Operator approximate diagonals from nets in $L^2(\G)$}

This section uses the machinery of quantum groups to discuss operator bounded approximate diagonals for $L^1(\G)$. 
We construct these diagonals by taking the simple tensors of elements of nets in $L^2(\G)$ satisfying (SA) and (CA) and then apply the multiplicative unitary of the commutant quantum group to the simple tensors.  This results in a net in $L^2(\G)\otimes_2L^2(\G)$.  The vector states associated to this net generate a net in $L^1(\G)\hat{\otimes}L^1(\G)$.

We begin this section by showing that the constructed net satisfies (OBAD1) for any locally compact quantum group.  For a group $G$ (OBAD2) is automatically satisfied so the above mentioned construction provides a novel method of describing a bounded approximate diagonal for $L^1(G)$.  This approach relies only on the LCQG structure, but does seem to require the equivalence of the multiplicative operators $W$ and $W'$.  This requirement can be weakened.  The operators only need to be approximately equivalent when applied to the relevant nets.  Continuing this approach, we show that if there are nets in $L^2(\G)$ (one which generates a bounded approximate identity,  and the other which demonstrates strong amenability) and $W^*W'$ applied to either of these nets approximates the identity then we can construct an operator bounded approximate diagonal.

The following lemmas will be useful in the sequal.  They are straightforward manipulations of the pentagonal rule and some of the other properties listed in proposition \eqref{prop:list}.

\begin{lemma}
\label{lem:pentagonal}
For any locally compact quantum group, the following equations involving multiplicative unitaries hold:
\[
W_{12}{W'}_{23}^* = W'^*_{23}W_{13}W_{12}
;
\]

\[
W_{23}{W'}_{12}^* = W'^*_{12}W'^*_{13}W_{23}
;
\]
and
\[
W^*_{13}W^*_{23} = (\hat{J}\otimes\hat{J}\otimes J)W_{13}W_{23}(\hat{J}\otimes\hat{J}\otimes J).
\]
\end{lemma}

\begin{proof}
Rewrite $W'$ and $W^*$ in terms of $W$ and the modular conjugations and apply the pentagonal rule.

\begin{align*}
W_{12}{W'}_{23}^* &= W_{12}(1\otimes J\otimes J)W_{23}^*(1\otimes J\otimes J)\\
&=W_{12}(\hat{J}\otimes J\otimes J)(\hat{J}\otimes 1\otimes 1)W_{23}^*(1\otimes J\otimes J)\\
&=(\hat{J}\otimes J\otimes J)W^*_{12}W_{23}^*(\hat{J}\otimes 1\otimes 1)(1\otimes J\otimes J)\\
&=(\hat{J}\otimes J\otimes J)W^*_{23}W^*_{13}W^*_{12}(\hat{J}\otimes J\otimes J)\\
&=W'^*_{23}W_{13}W_{12}.
\end{align*}

For the second result:
\begin{align*}
W_{23}{W'}_{12}^* &= W_{23}(J\otimes J \otimes 1)W_{12}^*(J\otimes J\otimes 1)\\
&=W_{23}(J\hat{J}\otimes 1 \otimes 1)(\hat{J}\otimes J \otimes 1)W_{12}^*(J\otimes J\otimes 1)\\
&=(J\hat{J}\otimes 1 \otimes 1)W_{23}W_{12}(\hat{J}\otimes J \otimes 1)(J\otimes J\otimes 1)\\
&=(J\hat{J}\otimes 1 \otimes 1)W_{12}W_{13}W_{23}(\hat{J}J\otimes 1 \otimes 1)\\
&=W'^*_{12}W'^*_{13}W_{23}.
\end{align*}

For the third result it is not necessary to apply the pentagonal rule, but only to notice that parts of simple tensors commute readily with leg tensors provided those parts are on different `legs'.
\begin{align*}
W^*_{13}W^*_{23} &= (\hat{J}\otimes 1 \otimes J)W_{13}(\hat{J}\otimes 1\otimes J)(1\otimes\hat{J}\otimes J)W_{23}(1\otimes \hat{J}\otimes J) \\
 &= (\hat{J}\otimes 1 \otimes J)W_{13}(\hat{J}\otimes \hat{J}\otimes 1)W_{23}(1\otimes \hat{J}\otimes J) \\
 &= (\hat{J}\otimes \hat{J} \otimes J)W_{13}W_{23}(\hat{J}\otimes \hat{J}\otimes J) 
\end{align*}
\end{proof}

The following lemma is an adaptation of \cite[3.14]{ruanxu}.

\begin{lemma}
\label{lem:ruanxu}
Given a unit vector $\xi\in L^2(\G)$, the map $\theta_\xi:L^\infty(\G){\otimes}_{VN}L^\infty(\G)\rightarrow B(L^2(\G))$ given by
\[
\theta_\xi(\Lambda) = (\omega_\xi\otimes i)(W' \Lambda  W'^*)
\]
(for $\Lambda\in L^\infty(\G)$) 
is weak* continuous, unital, and completely positive
.  Furthermore, $\theta_\xi$ maps $L^\infty(\G)\bar\otimes 
L^\infty(\G)$ into $L^\infty(\G)$.
\end{lemma}

\begin{proof}
Since $W' \in L^\infty(\G)'\otimes_{VN}L^\infty(\hat{\G})$, the first leg of $\Lambda$ commutes with $W'$.  For simple tensors $X\otimes Y\in L^\infty(\G)\otimes L^\infty(\G)$, 
\[
\theta_\xi(X\otimes Y) = (\omega_{J\hat{J}\xi}\otimes i)((X\otimes 1)\Gamma(Y))
\]
For complete details, see \cite[page 205]{ruanxu}.
\end{proof}

We now show that it is possible to combine nets in $L^2(\G)$ with properties (CA) and (SA) to create a net in $L^1(\G)\hat\otimes L^1(\G)$ which satisfies the first condition of an approximate diagonal.  The second condition is immediately satisfied if $W=W'$, but if this is not the case, then some additional assumption must be made.
\begin{theorem}
\label{thm:33}
Let $\G$ be a locally compact quantum group.  Let $\varepsilon>0$.  Let $\zeta\in L^2(\G)$ with $\|\zeta\|=1$.
Suppose that $\xi, \eta\in L^2(\G)_1$ such that 
\[
\| W(\zeta\otimes \xi) - (\zeta\otimes \xi)\|_2<\varepsilon;
\]

and
\[
\|\omega_\zeta\ast\omega_\eta - \omega_\eta\ast\omega_\zeta\|_1<\varepsilon.
\]

Then for $\Lambda \in L^\infty(\G)\otimes_{VN} L^\infty(\G)$
\[
\left|(\omega_\zeta\cdot \omega_{{W'}^*(\xi\otimes\eta)} - \omega_{{W'}^*(\xi\otimes\eta)}\cdot\omega_\zeta)\left(\Lambda\right)\right|<3\varepsilon \left\|\Lambda\right\|
\]
\end{theorem}

\begin{proof}
Consider
\[
\left|\left(\omega_\zeta\cdot \omega_{{W'}^*(\xi\otimes\eta)} - \omega_{{W'}^*(\xi\otimes\eta)}\cdot\omega_\zeta\right)\left(\Lambda\right)\right|.
\]

We begin by rewriting in terms of the inner product in $L^2(\G)$.

\begin{align*}
&\left|\left(\omega_\zeta\cdot \omega_{{W'}^*(\xi\otimes\eta)} - \omega_{{W'}^*(\xi\otimes\eta)}\cdot\omega_\zeta\right)\left(\Lambda\right)\right|\\
&=\left|\left(\omega_\zeta\otimes \omega_{{W'}^*(\xi\otimes\eta)}\right)\left(\left(\Gamma\otimes i\right)(\Lambda)\right)\right.\\
&\qquad\left.- \left(\omega_{{W'}^*(\xi\otimes\eta)}\otimes\omega_\zeta\right)\left((i\otimes\Gamma)(\Lambda)\right)\right|\\
&=\left|\left(\omega_\zeta\otimes \omega_{{W'}^*(\xi\otimes\eta)}\right)\left(W_{12}^*\Lambda_{23} W_{12}\right)\right.\\
&\qquad\left. - \left(\omega_{{W'}^*(\xi\otimes\eta)}\otimes\omega_\zeta\right)\left(W^*_{23}\Lambda_{13}W_{23}\right)\right|\\
&=\left|\left\langle \left(\Lambda_{23}\right)W_{12}{W'}_{23}^*(\zeta\otimes \xi\otimes\eta),  W_{12}{W'}_{23}^*(\zeta\otimes\xi\otimes\eta)\right\rangle \right.\\
&\qquad -\left. \left\langle \left(\Lambda_{13}\right)W_{23}{W'}_{12}^*(\xi\otimes\eta\otimes\zeta),  W_{23}{W'}_{12}^*(\xi\otimes\eta\otimes\zeta) \right\rangle  \right|
\end{align*}
By using the pentagonal rule results of Lemma \eqref{lem:pentagonal} this becomes:
\begin{align*}
\ldots&=\left|\left\langle \left(\Lambda_{23}\right)W'^*_{23}W_{13}W_{12}(\zeta\otimes \xi\otimes \eta), W'^*_{23}W_{13}W_{12}(\zeta\otimes \xi\otimes \eta)\right\rangle \right.\\
&\qquad -\left. \left\langle \left( \Lambda_{13}\right)W'^*_{12}W'^*_{13}W_{23}(\xi\otimes\eta\otimes\zeta), W'^*_{12}W'^*_{13}W_{23}(\xi\otimes\eta\otimes\zeta) \right\rangle  \right|
\end{align*}

We now have a $W(\zeta\otimes\xi)$ in the first term which is, by assumption, within $\varepsilon$ of $(\zeta\otimes\xi)$.  We subtract and add an appropriate term (notice the change in order of the tensors) and apply the triangle inequality to get:
\begin{align*}
\ldots&\leq\left|\left\langle (\Lambda_{23})W'^*_{23}W_{13}W_{12}(\zeta\otimes \xi\otimes \eta), W'^*_{23}W_{13}W_{12}(\zeta\otimes \xi\otimes  \eta)\right\rangle \right.\\
&\qquad -\left.\left\langle (\Lambda_{23})W'^*_{23}W_{13}(\zeta\otimes \xi\otimes \eta), W'^*_{23}W_{13}(\zeta\otimes \xi\otimes \eta)\right\rangle  \right|\\
&\qquad + \left|\left\langle (\Lambda_{13})W'^*_{13}W_{23}(\xi\otimes\zeta\otimes \eta), W'^*_{13}W_{23}(\xi\otimes\zeta\otimes \eta)\right\rangle   \right.\\
&\qquad -\left. \left\langle (\Lambda_{13})W'^*_{12}W'^*_{13}W_{23}(\xi\otimes\eta\otimes\zeta), W'^*_{12}W'^*_{13}W_{23}(\xi\otimes\eta\otimes\zeta) \right\rangle   \right|
\end{align*}

By the first assumption, the first difference in absolute values is less than $2\varepsilon\|\Lambda\|$.  The final term involves a $\Lambda_{13}$ and a $W'^*_{12}$.  These commute because their first legs are (respectively) in $L^\infty(\G)$ and its commutant.  So $W'_{12}\Lambda_{13} W'^*_{12} = \Lambda_{13}$ and we have:

\begin{align*}
\left|\left(\omega_\zeta\cdot \omega_{{W'}^*(\xi\otimes\eta)} \right.\right. &-\left.\left.\omega_{{W'}^*(\xi\otimes\eta)}\cdot\omega_\zeta\right)\left(\Lambda\right)\right|
\\
&< 2\varepsilon \left\|\Lambda\right\| \\
&\qquad +  \left|\left\langle (\Lambda_{13})W'^*_{13}W_{23}(\xi\otimes\zeta\otimes \eta), W'^*_{13}W_{23}(\xi\otimes\zeta\otimes \eta)\right\rangle   \right.\\
&\qquad -\left. \left\langle (\Lambda_{13})W'^*_{13}W_{23}(\xi\otimes\eta\otimes\zeta), W'^*_{13}W_{23}(\xi\otimes\eta\otimes\zeta) \right\rangle   \right|\\
& = 2\varepsilon \left\|\Lambda\right\| + \left|\left\langle (1\otimes \theta_{\xi}(\Lambda))W(\zeta\otimes  \eta),  W(\zeta\otimes  \eta)\right\rangle  \right.\\
&\qquad -\left. \left\langle (1\otimes \theta_{\xi}(\Lambda))W( \eta\otimes \zeta),  W(\eta\otimes \zeta) \right\rangle  \right|\\
& = 2\varepsilon \left\|\Lambda\right\| + \left|\left\langle ( \Gamma( \theta_{\xi}(\Lambda)))(\zeta\otimes  \eta),  ( \zeta\otimes  \eta)\right\rangle  \right.\\
&\qquad -\left. \left\langle (\Gamma(\theta_{\xi}(\Lambda)))( \eta\otimes \zeta),  (\eta\otimes \zeta) \right\rangle  \right|\\
&= 2\varepsilon \left\|\Lambda\right\|+ \left|\left(\theta_{\xi}(\Lambda)\right)\left((\omega_{\eta}\ast\omega_{\zeta})\right)\right.\\
&\qquad - \left.\left(\theta_{\xi}(\Lambda)\right)\left(\omega_{\zeta}\ast\omega_\eta)\right)\right|\\
&< 2\varepsilon \left\|\Lambda\right\| + \varepsilon\left\|\left(\theta_{\hat{J}J\xi}(\Lambda)\right)\right\|\\
&=3\varepsilon \|\Lambda\| 
\end{align*}

\end{proof}

\begin{corollary}
\label{cor:main}
Let $\G$ be a strongly amenable and co-amenable locally compact quantum group.  Suppose that $(\xi_\alpha)_\alpha, (\eta_\beta)_\beta$ are (SA) and (CA) nets in $L^2(\G)_1$ (respectively).  Suppose also that either, for $\zeta\in L^2(\G)$ :

\begin{equation}
\label{cond:C1}\|W^*(\xi_\alpha\otimes \zeta) - {W'}^*(\xi_\alpha \otimes \zeta )\|\rightarrow_\alpha 0;
\end{equation}
or
\begin{equation}
\label{cond:C2}\|W^*(\zeta\otimes \eta_\beta) - {W'}^*(\zeta \otimes \eta_\beta )\|\rightarrow_\beta 0.
\end{equation}
Then $\omega_{{W'}^*(\xi_\alpha\otimes \eta_\beta)}$ has a subnet which is an operator bounded approximate diagonal for $L^1(\G)\hat{\otimes} L^1(\G)$ hence $L^1(\G)$ is operator amenable.
\end{corollary}
\begin{proof}

Runde \cite{runde} showed that $\omega_{\eta_\beta}$ is a two sided bounded approximate identity for $L^1(\G)$.

For fixed $\varepsilon>0$, and $\zeta\in L^2(\G)_1$ there exist (by theorem \eqref{thm:33}) $\alpha_0$ and $\beta_0$ such that
\[
\left|(\omega_\zeta\cdot \omega_{{W'}^*(\xi_\alpha\otimes\eta_\beta)} - \omega_{{W'}^*(\xi_\alpha\otimes\eta_\beta)}\cdot\omega_\zeta)\right|<\varepsilon 
\]
for any $\alpha\succeq\alpha_0$ and $\beta\succeq\beta_0$.

Now, if  \eqref{cond:C1} is true, there is an $\alpha_1\succeq \alpha_0$ such that for any $\alpha\succeq\alpha_1$
\[
\|{W'}^*\xi_\alpha\otimes\eta_{\beta_0} - W^*\xi_\alpha\otimes\eta_{\beta_0} \|<\varepsilon.
\]

Since $\varepsilon$ and $\zeta$ are arbitrary, by choosing $\alpha$ after $\beta$ there is a subnet of $(\omega_{{W'}^*\xi_\alpha\otimes\eta_\beta})$ that is an operator bounded approximate diagonal for $L^1(\G)$.

If, instead, condition \eqref{cond:C2} is true, then an operator bounded approximate diagonal can be found by choosing $\beta$ after $\alpha$.
\end{proof}

\begin{remark}
If $(\xi_\alpha)_\alpha$ is a (SA) net that also satisfies condition \eqref{cond:C1} then it generates a bounded approximate identity for $L^1(\hat{\G})$, $\hat{\omega}_{\xi_\alpha}$.  This b.a.i. also satisfies
\[
\|\hat{\omega}_{\hat{W}^{op*}\hat{W} \zeta\otimes\xi_\alpha} - \hat{\omega}_{\zeta\otimes\xi_\alpha}\|\rightarrow 0
\]
for any state in $L^1(\hat{\G})$, $\hat{\omega}_\zeta$.  This property is related to the notion of quasicentral bounded approximate identities for locally compact groups as studied in \cite{losertrindler}, \cite{stokke} and others.
\end{remark}

\section{Dual Version: Operator Amenability of $L^1(\hat{\G})$}

For a locally compact group $G$, it is well known (eg \cite{effrosruan, paterson}) that  $G$ is amenable if and only if $L^1(G)$ is (operator) amenable if and only if $L^1(\hat{G})=A(G)$ is operator amenable.  It has been conjectured that $L^1(\G)$ is operator amenable if and only if $L^1(\hat{\G})$ is as well.  The conditions (CA) and (SA) are naturally dual to each other, which provides some additional justification for this conjecture.  However, recently Caspers, Lee, and Ricard \cite{caspersleericard} have shown that these two conditions are not sufficient for operator amenability of $L^1(\G)$.  In this section, we convert the main result of section 3 to its natural dual.  The dual versions of condition \eqref{cond:C1} and \eqref{cond:C2} interchange the multiplicative unitary for the commutant quantum group  ($\hat{W}'$) with that of the quantum group with opposite co-multiplication ($\widehat{W^{op}}$).  
To create a more useful dual version, one would hope that if $W'$ is approximately $W$ acting on some net in the fashion of (3.1) or (3.2) then $W^{op}$ is approximately $W$ acting on some other net as (4.3) or (4.4).  In his proof of the 
equivalence of amenability of $L^1(G)$ and operator amenability of $A(G)$, Ruan \cite{ruan} was able to use the fact that, in the group case, $W=W'$ to construct a net that worked appropriately with $W$ and $W^{op}$.  Such a nice result does not seem achievable in the general LCQG case, but we are able to make some progress in this direction.

We use an approach motivated in part by a result of Losert and Rindler \cite{losertrindler} whereby they construct, for an amenable group, an asymptotically central approximate identity.  We are able to combine the two nets of elements of $L^2(\G)$ (SA net and CA net) to create a bounded approximate identity which also has an approximately central property as well.

\begin{corollary}

Suppose that $(\xi_\alpha)_\alpha, (\eta_\beta)_\beta$ are nets in $L^2(\G)_1$ such that for every $\zeta\in L^2(\G)$ :

\begin{align}
\label{conditionAdual}
\| W(\zeta\otimes \xi_\alpha) - (\zeta\otimes \xi_\alpha)\|&\rightarrow 0;\\
\label{conditionBdual}
\|W(\eta_\beta\otimes\zeta)-(\eta_\beta\otimes\zeta)\|&\rightarrow 0;\\
\|W(\zeta\otimes\eta_\beta ) - {W}^{op}(\zeta\otimes\eta_\beta)\|&\rightarrow 0; \text{ or} \\
\|W( \xi_\alpha\otimes\zeta) - {W}^{op}(\xi_\alpha\otimes\zeta)\|&\rightarrow 0.
\end{align}
Then $\omega_{\widehat{W^{op}} ^* (\xi_\alpha\otimes \eta_\beta)}$ is a operator bounded approximate diagonal for $L^1(\hat{\G})\hat{\otimes} L^1(\hat{\G})$ hence $L^1(\hat{\G})$ is operator amenable.
\end{corollary}
\begin{proof}
This is a consequence of Corollary \eqref{cor:main}, but rephrased for the dual quantum group.  

If $(\xi_\alpha)_\alpha, (\eta_\beta)_\beta$ are (SA) and (CA) nets for $\G$ in $L^2(\G)_1$ (respectively) then they are also (CA) and (SA) nets for $\hat{\G}$ (respectively).  

Furthermore
\begin{equation}
\label{conditionhatC1}\hat{W}^*(\xi\otimes \zeta) =\sigma \left(W(\zeta\otimes\xi)\right)
\end{equation}
and
\begin{equation}
{\hat{W}^{'*}}(\xi \otimes \zeta ) =\sigma\left( W^{op}(\zeta\otimes\xi)\right)
\end{equation}
where $\sigma$ swaps the coordinates in the tensor product.

The results follow from corollary \eqref{cor:main}.
\end{proof}

We use several more lemmas that involve manipulating the multiplicative unitary operators.

\begin{lemma}
\label{lem:qc}
\[
{W'}_{13}^{\text{op}*}W'_{13}{W'}_{23}^{\text{op}*}W_{23} = {W'}_{12}^{*}{W'}_{23}^{\text{op}*}W'_{23}W'_{12}{W'}_{23}^*W_{23}
\]
\end{lemma}
\begin{proof}
Because $W'\in L^\infty(\G)'\otimes_{VN}L^\infty(\hat{\G})$ and ${W'}^{\text{op}}\in L^\infty(\G)'\otimes_{VN}L^\infty(\hat{\G})'$, it follows that $W'_{13}$ and ${W'}^{\text{op}*}_{23}$ commute.

Now, by the pentagonal equation:
\[W'_{13}W_{23}  = W'_{13}W'_{23}{W'}^*_{23}W_{23} = {W'}_{12}^{*}W'_{23}W'_{12}{W'}^*_{23}W_{23}\]
and similarly 
\[
{W'}_{13}^{\text{op}*}{W'}_{23}^{\text{op}*}=\hat{J}_3 J_3 W'_{13}W'_{23} J_3 \hat{J}_3 = \hat{J}_3 J_3   {W'}_{12}^{*}W'_{23} W'_{12}   J_3 \hat{J}_3 = 
{W'}_{12}^{*}{W'}_{23}^{\text{op}*}W'_{12}.      
\]

Combining the above two equations, we get the desired result:
\begin{align*}
{W'}_{13}^{\text{op}*}W'_{13}{W'}_{23}^{\text{op}*}W_{23} &={W'}_{13}^{\text{op}*}{W'}_{23}^{\text{op}*}W'_{13}W_{23}\\
&={W'}_{12}^{*}{W'}_{23}^{\text{op}*}W'_{12}     {W'}_{12}^{*}W'_{23}W'_{12}{W'}^*_{23}W_{23}\\
&= {W'}_{12}^{*}{W'}_{23}^{\text{op}*}W'_{23}W'_{12}{W'}^*_{23}W_{23}
\end{align*}
\end{proof}

\begin{lemma}
\label{lem:bai}
\[
W_{23}W_{12}{{W}'}_{12}^{\text{op}*} = W_{12}{{W}'}_{12}^{\text{op}*}W_{13}W_{23}{W'}_{13}^*
\]
\end{lemma}
\begin{proof}
By the pentagonal rule we immediately get:
\[
W_{23}W_{12}{{W}'}_{12}^{\text{op}*} = W_{12}W_{13}W_{23}{{W}'}_{12}^{\text{op}*}
\]

By introducing factors of $J$ and $\hat{J}$ in the second leg we get:
\begin{align*}
W_{12}W_{13}\hat{J}_2J_2{W'}^*_{23}W'_{12}J_2\hat{J}_2 
&=W_{12}W_{13}\hat{J}_2J_2W'_{12}{W'}^*_{23}{W'}^*_{13}J_2\hat{J}_2\\
&=W_{12}W_{13}{W'}_{12}^{\text{op}*}W_{23}{W'}_{13}^*
\end{align*}

Finally, note that the first legs of $W$ and ${W'}^{\text{op}*}$ commute, so we get the desired result:
\[
W_{23}W_{12}{{W}'}_{12}^{\text{op}*} = W_{12}{{W}'}_{12}^{\text{op}*}W_{13}W_{23}{W'}_{13}^*
\]
\end{proof}

\begin{theorem}

Suppose we have nets $(\xi_\alpha)_\alpha$ satisfying conditions (SA)  and \eqref{cond:C1} and $(\eta_\beta)_\beta$ satisfying (CA) and \eqref{cond:C2}.  
For each $\alpha, \beta$, consider the element $u_{\alpha, \beta}$ of $L^1(\G)$ given by:
\[
u_{\alpha, \beta} = (1\otimes \imath) \omega_{WW^{\text{op}'*}\xi_\alpha\otimes\eta_\beta}.
\]

Then there exists a subnet $u_\gamma = u_{\alpha_\gamma, \beta_\gamma}$ which is a bounded approximate identity for $L^1(\G)$ 
and satisfies the following quasi-central condition:
\begin{equation}
\label{cond:qc}
(\omega_\zeta\otimes u_\gamma)({W'}^*W^{\text{op}'}\Lambda W^{\text{op}'*}{W'} - \Lambda)\rightarrow 0
\end{equation}
for $\zeta\in L^2(\G)$ and $\Lambda\in L^\infty(\G)\otimes_{VN}L^\infty(\G)$.
\end{theorem}
\begin{proof}
For $X\in L^\infty(\G)$ and $\zeta\in L^2(\G)$, we work towards showing that $u_\gamma$ is a bounded approximate identity by considering the value of $X(u_{\alpha, \beta}\ast \omega_\zeta)$ .

By applying lemma \eqref{lem:bai} and noting that $X_3 = 1\otimes 1\otimes X$ commutes with anything in the first two components, we see that:
\begin{align*}
X&(u_{\alpha, \beta}\ast \omega_\zeta)\\
 &= \langle  X_3 W_{23} W_{12}W_{12}^{\text{op}'*}\xi_\alpha\otimes\eta_\beta\otimes\zeta, W_{23} W_{12}W_{12}^{\text{op}'*}\xi_\alpha\otimes\eta_\beta\otimes\zeta\rangle\\
&= \langle  X_3 W_{12}{{W}'}_{12}^{\text{op}*}W_{13}W_{23}{W'}_{13}^* \xi_\alpha\otimes\eta_\beta\otimes\zeta,    W_{12}{{W}'}_{12}^{\text{op}*}W_{13}W_{23}{W'}_{13}^*  \xi_\alpha\otimes\eta_\beta\otimes\zeta\rangle \\
&= \langle  X_3 W_{13}W_{23}{W'}_{13}^* \xi_\alpha\otimes\eta_\beta\otimes\zeta,    W_{13}W_{23}{W'}_{13}^*  \xi_\alpha\otimes\eta_\beta\otimes\zeta\rangle.
\end{align*}

Using this, and the triangle inequality, we see that:
\begin{align*}
|X(u_{\alpha, \beta}&\ast \omega_\zeta)  - X(\omega_\zeta)|\\
&= | \langle X_3 W_{13}W_{23}{W'}_{13}^* \xi_\alpha\otimes\eta_\beta\otimes\zeta,    W_{13}W_{23}{W'}_{13}^*  \xi_\alpha\otimes\eta_\beta\otimes\zeta\rangle - \langle X\zeta, \zeta\rangle| \\
&\leq \left| \langle  X_3 W_{13}W_{23}{W'}_{13}^* \xi_\alpha\otimes\eta_\beta\otimes\zeta,    W_{13}W_{23}{W'}_{13}^*  \xi_\alpha\otimes\eta_\beta\otimes\zeta\rangle \right. \\
&\quad \quad \left.- \langle  X_3 W_{13}{W'}_{13}^* \xi_\alpha\otimes\eta_\beta\otimes\zeta,    W_{13}{W'}_{13}^*  \xi_\alpha\otimes\eta_\beta\otimes\zeta\rangle\right| \\
&\quad +\left| \langle X_3 W_{13}{W'}_{13}^* \xi_\alpha\otimes\eta_\beta\otimes\zeta,    W_{13}{W'}_{13}^*  \xi_\alpha\otimes\eta_\beta\otimes\zeta\rangle \right.\\
&\quad\quad -\left. \langle X_3\xi_\alpha\otimes\eta_\beta\otimes\zeta, \xi_\alpha\otimes\eta_\beta\otimes\zeta\rangle\right| \\
&\leq 2\|X\|\left(\|W{W'}^*\xi_\alpha\otimes\zeta - \xi_\alpha\otimes\zeta \| \right.\\
&\quad \left.+ \|W_{23}{W'}^*_{13} \xi_\alpha\otimes\eta_\beta\otimes\zeta - {W'}^*_{13} \xi_\alpha\otimes\eta_\beta\otimes\zeta\|\right)
\end{align*}

Now we consider a fixed $\Lambda\in L^\infty(\G)\otimes_{VN} L^\infty(\G)$ and $\zeta\in L^2(\G)$ to show the quasi-central property.  By lemma \eqref{lem:qc} and since $\Lambda_{13}$, and $W'_{12}$ commute:
\begin{align*}
&(\omega_\zeta\otimes u_\gamma)({W'}^*W^{\text{op}'}\Lambda W^{\text{op}'*}{W'} - \Lambda) \\
&= \langle \Lambda_{13} W_{13}^{\text{op}'*}{W'}_{13}W_{23}^{\text{op}'*}W_{23}\zeta\otimes\xi_\alpha\otimes\eta_\beta, W_{13}^{\text{op}'*}{W'}_{13}W_{23}^{\text{op}'*}W_{23}\zeta\otimes\xi_\alpha\otimes\eta_\beta\rangle\\
&\quad  - \langle \Lambda_{13} W_{23}^{\text{op}'*}W_{23}\zeta\otimes\xi_\alpha\otimes\eta_\beta, W_{23}^{\text{op}'*}W_{23}\zeta\otimes\xi_\alpha\otimes\eta_\beta\rangle\\
&=\hspace{-.06cm}\langle \Lambda_{13} {W}_{12}^{'*}{W}_{23}^{\text{op}'*}W'_{23}W'_{12}{W}_{23}^{'*}W_{23}\zeta\hspace{-.06cm}\otimes\hspace{-.06cm}\xi_\alpha\hspace{-.06cm}\otimes\hspace{-.06cm}\eta_\beta, {W}_{12}^{'*}{W}_{23}^{\text{op}'*}W'_{23}W'_{12}{W}_{23}^{'*}W_{23}\zeta\hspace{-.06cm}\otimes\hspace{-.06cm}\xi_\alpha\hspace{-.06cm}\otimes\hspace{-.06cm}\eta_\beta\rangle\\
&\quad  - \langle \Lambda_{13} W_{23}^{\text{op}'*}W_{23}\zeta\otimes\xi_\alpha\otimes\eta_\beta, W_{23}^{\text{op}'*}W_{23}\zeta\otimes\xi_\alpha\otimes\eta_\beta\rangle\\
&= \langle \Lambda_{13} {W'}_{23}^{\text{op}*}W'_{23}W'_{12}{W'}_{23}^*W_{23}\zeta\otimes\xi_\alpha\otimes\eta_\beta, {W'}_{23}^{\text{op}*}W'_{23}W'_{12}{W'}_{23}^*W_{23}\zeta\otimes\xi_\alpha\otimes\eta_\beta\rangle\\
&\quad  - \langle \Lambda_{13} W_{23}^{\text{op}'*}W_{23}\zeta\otimes\xi_\alpha\otimes\eta_\beta, W_{23}^{\text{op}'*}W_{23}\zeta\otimes\xi_\alpha\otimes\eta_\beta\rangle
\end{align*}

So now we have:
\begin{align*}
|(&\omega_\zeta\otimes u_\gamma)({W'}^*W^{\text{op}'}\Lambda W^{\text{op}'*}{W'} - \Lambda)| \\
&=|\langle \Lambda_{13} {W'}_{23}^{\text{op}*}W'_{23}W'_{12}{W'}_{23}^*W_{23}\zeta\otimes\xi_\alpha\otimes\eta_\beta, {W'}_{23}^{\text{op}*}W'_{23}W'_{12}{W'}_{23}^*W_{23}\zeta\otimes\xi_\alpha\otimes\eta_\beta\rangle\\
&\quad\quad 
- \langle \Lambda_{13} W_{23}^{\text{op}'*}W_{23}\zeta\otimes\xi_\alpha\otimes\eta_\beta, W_{23}^{\text{op}'*}W_{23}\zeta\otimes\xi_\alpha\otimes\eta_\beta\rangle|\\
&\leq |\langle \Lambda_{13} {W'}_{23}^{\text{op}*}W'_{23}W'_{12}{W'}_{23}^*W_{23}\zeta\otimes\xi_\alpha\otimes\eta_\beta, {W'}_{23}^{\text{op}*}W'_{23}W'_{12}{W'}_{23}^*W_{23}\zeta\otimes\xi_\alpha\otimes\eta_\beta\rangle\\
&\quad\quad
-\langle \Lambda_{13} {W'}_{23}^{\text{op}*}W'_{23}W'_{12}\zeta\otimes\xi_\alpha\otimes\eta_\beta, {W'}_{23}^{\text{op}*}W'_{23}W'_{12}\zeta\otimes\xi_\alpha\otimes\eta_\beta\rangle |\\
&\quad+|\langle \Lambda_{13} W_{23}^{\text{op}'*}W'_{23}W'_{12}\zeta\otimes\xi_\alpha\otimes\eta_\beta, W_{23}^{\text{op}'*}W'_{23}W'_{12}\zeta\otimes\xi_\alpha\otimes\eta_\beta\rangle\\
&\quad \quad
- \langle\Lambda_{13} W_{23}^{\text{op}'*}W'_{23}\zeta\otimes\xi_\alpha\otimes\eta_\beta, W_{23}^{\text{op}'*}W'_{23}\zeta\otimes\xi_\alpha\otimes\eta_\beta\rangle|\\
&\quad +|\langle \Lambda_{13} W_{23}^{\text{op}'*}W_{23}{W}^*_{23}W'_{23}\zeta\otimes\xi_\alpha\otimes\eta_\beta, W_{23}^{\text{op}'*}W_{23}{W}^*_{23}W'_{23}\zeta\otimes\xi_\alpha\otimes\eta_\beta\rangle\\
&\quad \quad
- \langle \Lambda_{13} W_{23}^{\text{op}'*}W_{23}\zeta\otimes\xi_\alpha\otimes\eta_\beta, W_{23}^{\text{op}'*}W_{23}\zeta\otimes\xi_\alpha\otimes\eta_\beta\rangle|\\
&\leq 2\|\Lambda\|_{\infty}\left( \|{W'}_{23}^*W_{23}\zeta\otimes\xi_\alpha\otimes\eta_\beta - \zeta\otimes\xi_\alpha\otimes\eta_\beta\|_2\right.\\
&\quad\quad + \|W'_{12}\zeta\otimes\xi_\alpha\otimes\eta_\beta - \zeta\otimes\xi_\alpha\otimes\eta_\beta\|_2\\
&\left.\quad\quad + \|{W}^*_{23}W'_{23}\zeta\otimes\xi_\alpha\otimes\eta_\beta- \zeta\otimes\xi_\alpha\otimes\eta_\beta\|_2\right)
\end{align*}

For fixed $\varepsilon>0$, there exists an $\alpha_\varepsilon$ such that, by (SA)
\[
\|W'_{12}\zeta\otimes\xi_\alpha\otimes\eta_\beta - \zeta\otimes\xi_\alpha\otimes\eta_\beta\|_2
<\varepsilon
\]
and by \eqref{cond:C1}
\[
\|W{W'}^*\xi_\alpha\otimes\zeta - \xi_\alpha\otimes\zeta \|
<\varepsilon
\]
Furthmore, there is a $\beta_\varepsilon$ such that, by (CA)
\[
 \|W_{23}{W'}^*_{13} \xi_\alpha\otimes\eta_\beta\otimes\zeta - {W'}^*_{13} \xi_\alpha\otimes\eta_\beta\otimes\zeta\|
<\varepsilon
\]
and by \eqref{cond:C2}
\begin{align*}
 \|{W'}_{23}^*W_{23}\zeta\otimes\xi_\alpha\otimes\eta_\beta - \zeta\otimes\xi_\alpha\otimes\eta_\beta\|_2
&<\varepsilon \text{ and }\\
\|{W}^*_{23}W'_{23}\zeta\otimes\xi_\alpha\otimes\eta_\beta- \zeta\otimes\xi_\alpha\otimes\eta_\beta\|_2
&<\varepsilon
\end{align*}

So, by choosing $\beta$ after $\alpha$ there is a subnet $u_\gamma$ which is a bounded approximate identity satisfying condition \eqref{cond:qc}.
\end{proof}
Converting the result above for $L^1(\G)$ into the corresponding result for $L^2(\G)$ that is desired for the previous result may be difficult.  It is worth noting the similarity between this challenge and the
open problem of whether the existence of a left invariant mean implies strong amenability.  Runde and Daws conjectured that some version of Leptin's theorem would be helpful in the latter case.  It would perhaps be similarly helpful for the former.
\\
\\
{\bf Acknowledgement.} The author gratefully acknowledges the financial support of Hanyang University via a research fund for new professors.  

The author would like to thank Professor Zhiguo Hu and others at the University of Windsor for many helpful comments and guidance in this research.  

The author would like to thank the referee for his/her careful reading of the paper and helpful comments and corrections and Yemon Choi for directing the author to the paper of Ruan and Xu \cite{ruanxu}.

\bibliographystyle{amsplain}

\end{document}